\documentclass[12pt]{amsart}

\usepackage{xparse} 
\usepackage{tgpagella}
\usepackage{euler}
\usepackage[T1]{fontenc}
\usepackage{amsmath, amssymb}
\usepackage[hidelinks]{hyperref}
\usepackage[english]{babel}
\usepackage{mathrsfs}
\usepackage{eucal}
\usepackage[all]{xy}
\usepackage{tikz}

\newtheorem{thm}{Theorem}
\newtheorem*{thm*}{Theorem}
\newtheorem{lem}[thm]{Lemma}
\newtheorem{fact}[thm]{Fact}

\newtheorem{prop}[thm]{Proposition}
\newtheorem*{prop*}{Proposition}

\newtheorem{cor}[thm]{Corollary}
\newtheorem*{cor*}{Corollary}

\theoremstyle{definition}
\newtheorem{defn}[thm]{Definition}
\newtheorem*{defn*}{Definition}

\newtheorem{examples}[thm]{Examples}

\newtheorem*{question*}{Question}
\newtheorem*{Pquestion*}{Popa's question}

\newtheorem*{conv*}{Convention}

\newcommand{\norm}[1]{{\left\lVert #1\right\rVert}}

\def\bb{\mathbb}

\def\bb{\mathbb}

\def\cal{\mathcal}

\def\u{\mathsf 1}

\makeatletter

\def\dotminussym#1#2{%
  \setbox0=\hbox{$\m@th#1-$}%
  \kern.5\wd0%
  \hbox to 0pt{\hss\hbox{$\m@th#1-$}\hss}%
  \raise.6\ht0\hbox to 0pt{\hss$\m@th#1.$\hss}%
  \kern.5\wd0}

\def \R{\mathcal R}
\def \u{\mathcal U}
\def \val{\operatorname{val}}

\newcommand{\mc}{\mathcal}
\newcommand{\mb}{\mathbb}

\newcommand{\cstar}{$\mathrm{C}^*$}

\newcommand{\mip}{\operatorname{MIP}}
\newcommand{\op}{\operatorname{op}}

\def \u {\mathcal U}

 \NewDocumentCommand{\urep}{omm}{
  \IfNoValueTF{#1}
  {\mathrm{C}^{\mbox{$*$}} \langle#2\;|\;#3\rangle}
  {\mathrm{C}^{\mbox{$*$}}\mathopen{#1\langle}#2\;#1|\;#3\mathclose{#1\rangle}}
}%
 
\NewDocumentCommand{\ugcsa}{om}{
  \IfNoValueTF{#1}
  {\mathrm{C}^{\mbox{$*$}} (#2)}
  {\mathrm{C}^{\mbox{$*$}}\mathopen{#1(}#2\mathclose{#1)}}
}%

\begin{document}


\title{Locally universal \cstar-algebras with computable presentations}

\author{Alec Fox, Isaac Goldbring, and Bradd Hart}
\thanks{Goldbring was partially supported by NSF grant DMS-2054477. Hart was funded by the NSERC.}

\address{Department of Mathematics\\University of California, Irvine, 340 Rowland Hall (Bldg.\# 400),
Irvine, CA 92697-3875}
\email{foxag@uci.edu}
\urladdr{https://www.math.uci.edu/people/alec-fox}
\email{isaac@math.uci.edu}
\urladdr{http://www.math.uci.edu/~isaac}

\address{Department of Mathematics and Statistics, McMaster University, 1280 Main St., Hamilton ON, Canada L8S 4K1}
\email{hartb@mcmaster.ca}
\urladdr{http://ms.mcmaster.ca/~bradd/}

\maketitle

\begin{abstract}
The Kirchberg Embedding Problem (KEP) asks if every \cstar-algebra embeds into an ultrapower of the Cuntz algebra $\cal O_2$.  In an effort to provide a negative solution to the KEP and motivated by the recent refutation of the Connes Embedding Problem, we establish two computability-theoretic consequences of a positive solution to KEP.  Both of our results follow from the a priori weaker assumption that there exists a locally universal \cstar-algebra with a computable presentation.
\end{abstract}

\section{Introduction}

The recent landmark quantum complexity result known as  $\mip^*=\operatorname{RE}$ \cite{MIP*} yielded a negative solution to a famous problem in the theory of von Neumann algebras, namely the \textbf{Connes Embedding Problem} (CEP). The CEP, posed in Connes' seminal paper \cite{connes}, asks if every tracial von Neumann algebra embeds into a tracial ultrapower of the hyperfinite II$_1$ factor.  The negative solution to the CEP can be used to give a negative solution to an analogous problem in the theory of \cstar-algebras known as the \textbf{Blackadar-Kirchberg Problem} (or \textbf{MF Problem}), which asked if every stably finite \cstar-algebra embeds into an ultrapower of the universal UHF algebra (see \cite[Proposition 6.1]{universal}).

The Blackadar-Kirchberg Problem can be viewed as the ``finite'' \cstar-algebra analog of CEP.  In this paper, we consider the ``infinite'' \cstar-algebra analog of CEP known as the \textbf{Kirchberg Embedding Problem} (KEP). KEP asks if every \cstar-algebra embeds into an ultrapower of the Cuntz algebra $\cal O_2$.\footnote{Incidentally, one might ask if there is an ``infinite'', that is, type III, von Neumann algebra analog of CEP.  For example, one might ask if every von Neumann algebra embeds with expectation into the Ocneanu ultrapower of the hyperfinite III$_1$ factor $\R_\infty$.  In \cite{AHW}, Ando, Haagerup, and Winslow prove that this question is equivalent to CEP itself (and thus has a negative answer).} 

The Kirchberg Embedding Problem was studied model theoretically by the second author and Sinclair in \cite{GS}.  In that paper, KEP is shown to be equivalent to the statement that there is a \cstar-algebra that is both nuclear and \textbf{existentially closed}.  Further model-theoretic equivalents of KEP were given by the second author in \cite{games}, where it was shown that  KEP is equivalent to the statement that $\cal O_2$ is the enforceable \cstar-algebra.

In this paper, we take our attention away from the nuclearity of $\cal O_2$ and instead focus on the fact that it has a so-called \textbf{computable presentation}.  Here, a presentation of a \cstar-algebra is computable if there is an algorithm that can effectively approximate the norm of all rational polynomials applied to the generators.  As mentioned in Fact \ref{alec} below, the standard presentation of $\cal O_2$ is computable.  We then derive two computability-theoretic consequences of the existence of a \textbf{locally universal} \cstar-algebra with a computable presentation, where here, a \cstar-algebra $A$ is locally universal if every \cstar-algebra embeds into an ultrapower of $A$.  We believe that neither of these consequences should hold generally and thus view these results as evidence that the KEP has a negative solution.



We now briefly describe our two main results, giving more precise statements in the two sections that follow.

Our first theorem concerns the notion of a \textbf{weakly stable} \cstar-algebra.  Roughly speaking, a finitely presented \cstar-algebra $A$ is weakly stable if any almost $*$-homomorphisms of $A$ into a \cstar-algebra $B$ is near an actual $*$-homomorphism of $A$ into $B$.  Somewhat more precisely, the finitely presented \cstar-algebra $$A=C^*\langle \bar x \ |\ p_j(x_1,\ldots,x_k)=0 \ , j=1,\ldots,m\rangle$$ (see the next section for a precise definition of this notion) is weakly stable if:  for any $\epsilon>0$, there is $\delta>0$ so that, for any \cstar-algebra $B$ and elements $z_1,\ldots,z_k\in B$ for which $\|p_j(\bar z)\|<\delta$ for all $j=1,\ldots,m$, there is a $*$-homomorphism $\varphi:A\to B$ for which $\|\varphi(x_i)-z_i\|\leq \epsilon$ for all $i=1,\ldots,k$.  Weakly stable \cstar-algebras (and other variants) have been studied by other researchers, perhaps most extensively in the foundational papers \cite{BShape} and \cite{loring}.  If we demand that this assignment $\epsilon\mapsto \delta$ is computable (in a sense we make precise in Definition \ref{cws}) below, we arrive at the notion of a \textbf{computably weakly stable} \cstar-algebra.  (To be precise, the property of being computably weakly stable is a property of a given finite presentation of a \cstar-algebra, but is preserved under computable isomorphism; see Lemma \ref{robust} below.)

Our first theorem states that if a locally universal \cstar-algebra with a computable presentation exists, then any computably weakly stable \textbf{finitely c.e. presentation} (meaning that the finitely many polynomials in the presentation have their coefficients from $\bb Q(i)$) of a \cstar-algebra is actually computable.  In short:  under the assumption that a locally universal \cstar-algebra with a computable presentation exists, an effective weak stability condition on the \cstar-algebra implies that one can effectively approximate the operator norm in the \cstar-algebra.

Our second result belongs to the area of quantum complexity theory. 
In \cite{MIP*}, a particular quantum complexity class $MIP^{co}$ is defined and it is shown that all of the languages in this class are coRE - complements of recursively enumerable sets.  It is left as an open question if the converse is true, namely, does $MIP^{co}$ coincide with the class coRE? In the notation $MIP^{co}$, the ``co'' stands for ``commuting'' and comes from the use of \textbf{commuting operator strategies} in its definition.  In this paper, we define a relaxed version $MIP^{co}_{\delta,op}$ of ``almost-commuting'' operator strategies and show that, assuming the existence of a locally universal \cstar-algebra with a computable presentation, all languages which belong to the corresponding quantum complexity class $MIP^{co}_{op,\delta}$ are actually recursively enumerable (and thus decidable!).

We end the introduction with some notation and terminology concerning ultraproducts of \cstar-algebras.  Throughout this paper, all \cstar-algebras are assumed to be unital and all $*$-homomorphisms are assumed to preserve the unit.  Moreover, $\u$ always  denotes a nonprincipal ultrafilter on $\bb N$.  Given a \cstar-algebra $C$, its \textbf{ultrapower with respect to $\u$}, denoted $C^\u$, is the quotient of the Banach algebra $\ell^\infty(\bb N,C)$ of all uniformly norm-bounded sequences from $C$ by the ideal of elements $(A_m)_{m\in \bb N}$ for which $\lim_\u \|A_m\|=0$.  It is well-known that $C^\u$ is a \cstar-algebra once again.  Given $(A_m)_{m\in \bb N}\in \ell^\infty(\bb N,C)$, we denote its coset in $C^\u$ by $(A_m)_\u$.


The authors would like to thank William Slofstra, Thomas Vidick and Henry Yuen for useful discussions regarding this work.

\section{Computable weak stability}

Let $\cal G$ be a set of noncommuting indeterminates, which we call \textbf{generators}.  By a set of \textbf{relations} for $\cal G$ we mean a set of relations of the form $\|p(x_1,\ldots,x_n)\|\leq a$, where $p$ is a $*$-polynomial in $n$ noncommuting variables with no constant term, $x_1,\ldots,x_n$ are elements of $\cal G$, and $a$ is a nonnegative real number.  We also require that, for every generator $x\in \cal G$, there is a relation of the form $\|x\|\leq M$ in $\cal R$.  A \textbf{representation} of $(\cal G,\cal R)$ is a function $j:\cal G\to A$, where $A$ is a \cstar-algebra, such that $\|p(j(x_1),\ldots,j(x_n))\|\leq a$ for every relation $\|p(x_1,\ldots,x_n)\|\leq a$ in $\cal R$.

The \textbf{universal \cstar-algebra} of $(\cal G,\cal R)$ is a \cstar-algebra $A$ along with a representation $\iota:\cal G\to A$ of $(\cal G,\cal R)$ such that, for all other representations $j:\cal G\to B$ of $(\cal G,\cal R)$, there is a unique *-homomorphism $\varphi:A\to B$ such that $\varphi(\iota(x))=j(x)$ for all $x\in \cal G$.  If the universal \cstar-algebra of $(\cal G,\cal R)$ exists, then it is unique up isomorphism and will be denoted by $C^*\langle \cal G|\cal R\rangle$.  Note that $C^*\langle \cal G|\cal R\rangle$ is generated by the image of the generators.  If $\cal G$ is a sequence $\bar x$, then we may write $C^*\langle \bar x | \cal R\rangle$ instead of $C^*\langle \bar G | \cal R\rangle$.  Given that we remain in the context of unital \cstar-algebras throughout this paper, we implicitly assume that we have a distinguished generator for the unit and include relations stating that it is a self-adjoint idempotent which acts as a multiplicative identity.

A \cstar-algebra $A$ is \textbf{finitely presented} if it is of the form $A=C^*\langle \cal G|\cal R\rangle$ for finite sets $\cal G$ and $\cal R$.

Let $C$ be a \cstar-algebra.  A \textbf{presentation} of $C$ is a pair $C^\dagger:=(C,(a_n)_{n\in \mathbb N})$, where $\{a_n \ : \ n\in \mathbb N\}$ is a subset of $C$ that generates $C$ (as a \cstar-algebra).  Elements of the sequence $(a_n)_{n\in \mathbb N}$ are referred to as \textbf{special points} of the presentation while elements of the form $p(a_{i_1},\ldots,a_{i_k})$ for $p$ a $*$-polynomial with coefficients from $\bb Q(i)$ (a \textbf{rational polynomial}) are referred to as \textbf{rational points} of the presentation.  We say that $C^\dagger$ is a \textbf{computable presentation} of $C$ if there is an algorithm such that, upon input a rational point $p$ of $C^\dagger$ and $k\in \mathbb N$, returns a rational number $q$ such that $|\|p\|-q|<2^{-k}$.
If $A^\dagger$ and $C^\dagger$ are presentations of \cstar-algebras, we say a *-homomorphism $\varphi : A \to C$ is a \textbf{computable map} from $A^\dagger$ to $C^\dagger$ if there is is an algorithim such that, upon input of a rational point $p$ of $A^\dagger$ and $k \in \mb{N}$, returns a rational point $q$ of $C^\dagger$ such that $\norm{\varphi(p) - q} < 2^{-k}$.

The \textbf{standard presentation} of a universal \cstar-algebra $C^*\langle \bar x|\cal R\rangle$ has $\bar x$ as its distinguished generating set.  A relation $\|p(x_1,\ldots,x_n)\|\leq a$ is called \textbf{rational} if $p$ is a rational polynomial and $a$ is a nonnegative dyadic rational.  A presentation $A^\dagger$ of a \cstar-algebra $A$ is called \textbf{c.e.} if it is the standard presentation $C^*\langle \bar x|\cal R\rangle$ of a c.e. set of rational relations $\cal R$.  When $\bar x$ and $\cal R$ are both finite, we say that $A^\dagger$ is \textbf{finitely c.e.}

We will need the following fact of the first-named author \cite{fox}:

\begin{fact}\label{alec}
If $A$ is a simple \cstar-algebra, then any c.e. presentation $A^\dagger$ of $A$ is computable.
\end{fact}

In particular, given that the standard presentation of $\cal O_2$ is clearly (finitely) c.e., we conclude from the previous fact that it is a computable presentation.

It is sensible to write a relation of the form $\|p(x_1,\ldots,x_n)\|\leq 0$ in the more familiar form $p(x_1,\ldots,x_n)=0$.  Note also that in any c.e. presentation of a \cstar-algebra, we can replace an arbitrary relation of the form $\norm{p(x_1,\ldots,x_n)} \leq a$ by the two relations $p(x_1,\ldots,x_n) - g = 0$ and $\norm{g} \leq a$, where $g$ is a new generator.

\begin{defn}
    A finitely presented \cstar-algebra $A = C^*\langle \overline{x} \;|\; p_i(\overline{x}) = 0, \norm{x_k} \leq C_k\rangle$ is \textbf{weakly stable} if for every $\epsilon > 0$ there is a $\delta > 0$ such that for all \cstar-algebras $B$ if $\overline{z} \in B$ satisfies $\norm{p_i(\overline{z})} \leq \delta$ for all $i$ and $\norm{z_k} \leq C_k + \delta$ for all $k$, then there exists a *-homorphism $\varphi : A \to B$ such that $\norm{\varphi(x_k) - z_k} \leq \epsilon$ for all $k$.
\end{defn}

The following are all examples of weakly stable \cstar-algebras; see \cite{BOp} for proofs.

\begin{examples}

\

\begin{enumerate}
    \item $\ugcsa{\mb{F}_n}$, the full universal \cstar-algebra of the free group on $n$ generators.
    \item $M_n(\mb{C})$, or, more generally, any finite-dimensional \cstar-algebra.
    \item Cuntz algebras $\mc{O}_n = \urep{s_1,\ldots,s_n}{s_i^*s_j = \delta_{ij}1, \sum s_is_i^* = 1}$.
    \item The Toeplitz algebra $\mc{T}$, the universal \cstar-algebra generated by an isometry.
\end{enumerate}
\end{examples}

Notice that in all of the above examples, the \cstar-algebra is either nuclear or else not simple.  In fact, we have:

\begin{prop}
If KEP has a positive solution, then any simple weakly stable \cstar-algebra is exact.
\end{prop}

\begin{proof}
Suppose that $A$ is a simple, weakly stable \cstar-algebra.  By a positive solution to KEP, there is an embedding $A\hookrightarrow \cal O_2^\u$.  Since $A$ is weakly stable (recall that $A$ is finitely presented and hence there are only finitely many conditions to be satisfied), there is a *-homomorphism $A\to \cal O_2$, which, since $A$ is simple, must be an embedding.  It follows that $A$ is exact.
\end{proof}

We now introduce a computable version of weak stability.

\begin{defn}\label{cws}
    A finitely c.e. presentation $A^\dagger = C^*\langle \overline{x} \;|\; p_i(\overline{x}) = 0, \norm{x_k} \leq C_k\rangle$ is \textbf{computably weakly stable} if there is an algorithm which when given $n \in \mathbb{N}$ returns $m \in \mathbb{N}$ such that for all \cstar-algebras $B$ if $\overline{z} \in B$ satisfies $\norm{p_i(\overline{z})} \leq 2^{-m}$ for all $i$ and $\norm{z_k} \leq C_k + 2^{-m}$ for all $k$, then there exists a *-homorphism $\varphi : A \to B$ such that $\norm{\varphi(x_k) - z_k} < 2^{-n}$ for all $k$.
\end{defn}

While the definition as written depends on the choice of relations, we show in the following lemma the notion is computably robust, so any finite set of rational relations that gives the same presentation will work.

\begin{lem}\label{robust}
    Let $$A^\dagger= \urep{x_1,\ldots,x_j}{p_i(\overline{x}) = 0, \norm{x_k} \leq C_k}$$ and $$A^\#= \urep{y_1,\ldots,y_\ell}{q_i(\overline{y}) = 0, \norm{y_k} \leq D_k}$$ be finitely c.e. presentations of a \cstar-algebra $A$.
    If there exists an isomorphism computable from $A^\#$ to $A^\dagger$ and $A^\dagger$ is computably weakly stable, then so is $A^\#$.
\end{lem}

\begin{proof}

    Fix a computable isomorphism $\psi$ from $A^\#$ to $A^\dagger$.
    Let $n \in \mb{N}$.
    For each $k = 1,\ldots,\ell$, we can compute a rational *-polynomial $t_k(\overline{x})$ such that $\norm{t_k(\overline{x}) - \psi(y_k)} < 2^{-(n+2)}$.
    Let $d$ be such that for any \cstar-algebra $B$ if $\norm{u_i - v_i}_B < 2^{-d}$ for $i = 1,\ldots,j$ then $\norm{t_k(\overline{u}) - t_k(\overline{v})}_B < 2^{-(n+2)}$ for all $k$.
    Let $s$ witness that $A^\dagger$ is computably weakly stable on input $d$.
    
    For $m \in \mb{N}$, let $W_m = \urep{w_1,\ldots,w_\ell}{\norm{q_i(\overline{w})} \leq 2^{-m}, \norm{w_k} \leq D_k + 2^{-m}}$.
    By \cite[Theorem 3.3]{fox}, given $m \in \mb{N}$ and a rational *-polynomial $r(\overline{w})$ we can effectively enumerate a decreasing sequence of rationals that converges to $\norm{r(\overline{w})}_{W_m}$.
    Enumerate over all $m \in \mb{N}$ and rational *-polynomials $r_1(\overline{w}),\ldots,r_j(\overline{w})$ and accept if $\norm{p_i(r_1(\overline{w}),\ldots,r_j(\overline{w}))}_{W_m} < 2^{-s}$ for all $i$ and $\norm{r_k(\overline{w})}_{W_m} < C_k + 2^{-s}$ for all $k$ and $\norm{t_k(r_1(\overline{w}),\ldots,r_j(\overline{w})) - w_k}_{W_m} < 2^{-(n+2)}$ for all $k$.
    Note an acceptance must happen since there is an embedding from $A$ into $\prod_\mc{U} W_m$ which sends $\psi(y_k)$ to $w_k$. 
    Hence, for any $\epsilon > 0$ there exist rational *-polynomials $r_1(\overline{w}),\ldots,r_j(\overline{w})$ in $\prod_\mc{U} W_m$ such that $\norm{p_i(r_1(\overline{w}),\ldots,r_j(\overline{w}))} < \epsilon$ for all $i$ and $\norm{r_k(\overline{w})} < C_k + \epsilon$ for all $k$ and $\norm{t_k(r_1(\overline{w}),\ldots,r_j(\overline{w}))- w_k} < \epsilon$ for all $k$.
    
    Let $B$ be a \cstar-algebra and $\overline{z} \in B$ such that $\norm{q_i(\overline{z})} \leq 2^{-m}$ for all $i$ and $\norm{z_k} \leq D_k + 2^{-m}$ for all $k$.
    Then $\norm{p_i(r_1(\overline{z}),\ldots,r_j(\overline{z}))} < 2^{-s}$ for all $i$ and $\norm{r_k(\overline{z})} < C_k + 2^{-s}$ for all $k$ and $\norm{t_k(r_1(\overline{z}),\ldots,r_j(\overline{z})) - z_k} < 2^{-(n+2)}$ for all $k$.
    By the choice of $s$, there exists $\varphi : A \to B$ such that $\norm{\varphi(x_k) - r_k(\overline{z})} < 2^{-d}$ for $k = 1,\ldots,j$.
    Then for $k = 1,\ldots,\ell$,
    \begin{align*}
        &\norm{\varphi(\psi(y_k)) - z_k}\\
        &\leq \norm{\varphi(\psi(y_k)) - \varphi(t_k(\overline{x}))}  \\
        &\qquad+ \norm{\varphi(t_k(\overline{x})) - t_k(r_1(\overline{z}),\ldots,r_j(\overline{z}))} \\
        &\qquad+ \norm{t_k(r_1(\overline{z}),\ldots,r_j(\overline{z})) - z_k}\\
        &\leq \norm{\psi(y_k) - t_k(\overline{x})}\\
        &\qquad+ \norm{t_k(\varphi(x_1),\ldots,\varphi(x_j)) - t_k(r_1(\overline{z},\ldots,r_j(\overline{z}))}\\
        &\qquad+ \norm{t_k(r_1(\overline{z}),\ldots,r_j(\overline{z})) - z_k}\\
        &< 2^{-n}.
    \end{align*}
\end{proof}

We next aim to show that all of the above examples of weakly stable \cstar-algebras are actually computably weakly stable.

The following result is folklore.
\begin{lem}\label{aunu}
    Suppose that $0 < \delta < \epsilon < 1$.  Then for any unital \cstar-algebra $A$ and $a \in A$, if $\norm{a^*a - 1} \leq \delta$ and $\norm{aa^* - 1} \leq \delta$, then there is a unitary $u \in A$ such that $\norm{a - u} < \epsilon$.
\end{lem}
Non-effective versions of the following results can be found in \cite{BShape} or \cite{loring}.
\begin{lem}\label{apnp}
    Suppose that $0 < \epsilon < 1$ and $0 < \delta < \epsilon^2/8$.  Then for any \cstar-algebra $A$ and $a \in A$, if $\norm{a} \leq 2$, $\norm{a - a^*} \leq \delta$, and $\norm{a - a^2} \leq \delta$, then there is a projection $p \in A$ such that $\norm{a - p} < \epsilon$.
\end{lem}
\begin{proof}
    Let $x = (a + a^*)/2$, so $x$ is self-adjoint and $\norm{a - x} \leq \delta/2 < \epsilon/2$.
    Also, by expanding, we see $\norm{x - x^2} \leq \norm{a-a^2} + \norm{a}\norm{a - a^*}/2 \leq 2\delta < \epsilon^2/4$.
    Hence the spectrum of $x$, $\sigma(x) \subseteq [-\epsilon/2, \epsilon/2] \cup [1 - \epsilon/2, 1 + \epsilon/2]$.
    Let $f$ be continuous on $\sigma(x)$ such that $f = 0$ on $[-\epsilon/2, \epsilon/2]$ and $f = 1$ on $[1 - \epsilon/2, 1 + \epsilon/2]$.
    Let $p = f(x)$.
    Then $p$ is a projection and $\norm{a - p} \leq \norm{a - x} + \norm{x - f(x)} < \epsilon$.
\end{proof}

\begin{cor}\label{directsum}
    If $A_1^\dagger$ and $A_2^\dagger$ are computably weakly stable, then so is $A_1^\dagger \oplus A_2^\dagger$.
\end{cor}
\begin{lem}\label{apinpi}
    Suppose $0 < \epsilon < 1$ and $0 < \delta < 2^{-16}\epsilon^8$. Then for any \cstar-algebra $A$, $a \in A$, and projections $p_1,p_2 \in A$, if $\norm{a^*a - p_1} \leq \delta$ and $\norm{aa^* - p_2} \leq \delta$, then there is a partial isometry $v$ such that $\norm{a - v} < \epsilon$, $v^*v = p_1$, and $vv^* = p_2$.
\end{lem}
\begin{proof}
    Let $b = p_2ap_1$.
    We have $\norm{a(1 - p_1)}^2 = \norm{(1 - p_1)(a^*a - p_1)(1 - p_1)} \leq \delta$, and similarly $\norm{(1 - p_2)a}^2 = \norm{(1 - p_2)aa^*(1 - p_2)} \leq \delta$.
    Hence $\norm{a - b} \leq \norm{a(1 - p_1)} + \norm{(1 - p_2)a} \leq 2\sqrt{\delta}$.
    Furthermore, \begin{align*}
            &\norm{(b^*b)^2 - b^*b} \\
            &\leq \norm{(b^*b)^2 - (a^*a)^2} + \norm{(a^*a)^2 - p_1^2} + \norm{p_1 - a^*a} + \norm{a^*a - b^*b} \\
            &\leq 20\norm{b - a} + 4\norm{a^*a - p_1} \\
            &\leq 49\sqrt{\delta}.
        \end{align*}
    Let $\gamma = 7\delta^{1/4}$.
    Then $\sigma(b^*b)$ is a subset of $[-\gamma,\gamma] \cup [1-\gamma,1+\gamma]$.
    Let $f$ be a continuous function on $\sigma(b^*b)$ such that $f = 0$ on $[-\gamma, \gamma]$ and $f(x) = \sqrt{x}$ for all $x \in [1-\gamma, 1+\gamma]$.
    Let $w = bf(b^*b)$.
    Then $w^*w = b^*bf(b^*b)^2$ and $ww^* = bb^*f(bb^*)^2$ are projections.
    Also, $\norm{b - w}^2 = \norm{(b^* - w^*)(b-w)} = \norm{b^*b - w^*w} \leq \gamma$, so $\norm{a - w} \leq \norm{a - b} + \norm{b - w} \leq 2\sqrt{\delta} + \sqrt{\gamma} < \epsilon$.
    Note $w^*w \leq p_1$ and $ww^* \leq p_2$.
    Then $p_1 - w^*w$ is a projection with norm 
    \[\norm{p_1 - w^*w} \leq \norm{p_1 - a^*a} + \norm{a^*a - b^*b} + \norm{b^*b - w^*w} < 1\] and $p_2 - ww^*$ is a projection with norm \[\norm{p_2 - ww^*} \leq \norm{p_2 - aa^*} + \norm{aa^* - bb^*} + \norm{bb^* - ww^*} < 1.\]
    Thus $w^*w = p_1$ and $ww^* = p_2$.

\end{proof}
\begin{cor}\label{macws}
    If $A^\dagger$ is computably weakly stable, then so is $M_n(A^\dagger)$.
\end{cor}

We now have that the following standard presentations are computably weakly stable:

\begin{examples}

\

\begin{enumerate}
    \item $\ugcsa{\mb{F}_n}$. 
    \item $\mb{C}$. More generally, the universal \cstar-algebra generated by $n$ projections.
    \item $M_n(\mb{C})$. More generally, any finite-dimensional \cstar-algebra.
    \item Cuntz algebras $\mc{O}_n$.
    \item The Toeplitz algebra $\mc{T}$.
\end{enumerate}
\end{examples}

For (1), apply Lemma \ref{aunu}; for (2), apply Lemma \ref{apnp}; for (3), apply Corollary \ref{directsum} and Lemma \ref{apinpi}; and for (4), apply Lemma \ref{apnp} and Lemma \ref{apinpi}.

Note that the previous presenations are all actually computable.  The following is the main result of this section:

\begin{thm}
    If there exists a locally universal \cstar-algebra $B$ with a computable presentation $B^\dagger$, then every computably weakly stable presentation $A^\dagger$ of a \cstar-algebra $A$ is computable.
\end{thm}
\begin{proof}
    Since $A^\dagger$ is c.e., by \cite[Theorem 3.3]{fox} there is an effective procedure which, when given a rational point $q$ of $A^\dagger$, enumerates a decreasing sequence of rationals that converges to $\norm{q}$.
    It is enough to show we can do the same from below.
    
    Given a rational point $q(\overline{x})$ of $A^\dagger$, we proceed as follows.
    Enumerate through all $j \in \mathbb{N}$.
    For each $j$, determine $n \in \mathbb{N}$ such that if $\norm{x_k - y_k} < 2^{-n}$ for all $k$ then $\norm{q(\overline{x}) - q(\overline{y})} < 2^{-j}$ in all \cstar-algebras.
    Let $m \in \mathbb{N}$ be given, upon input $n$, by the effective procedure which witnesses that $A^\dagger$ is computably weakly stable.
    Enumerate all tuples $\overline{r}$ of rational points of $B^\dagger$ such that $\norm{p_i(\overline{r})} < 2^{-m}$ for all $i$ and $\norm{r_k} < C_k + 2^{-m}$ for all $k$.
    Enumerate over all dyadic rationals $d$.
    If ever $\norm{q(\overline{r})} > 2^{-j} + d$ where $d$ is greater than all previous outputs, then output $d$.
    By the choice of $m$, there exists a *-homorphism $\varphi : A \to B$ such that $\norm{\varphi(x_k) - r_k} < 2^{-n}$.
    Then $\norm{q(\overline{x})} \geq \norm{\varphi(q(\overline{x}))} \geq \norm{q(\overline{r})} - 2^{-j} > d$.

    We show this sequence does converge to $\norm{q}$ from below.
    Indeed, let $d < \norm{q(\overline{x})}$ and let $j \in \mathbb{N}$ be such that $\norm{q(\overline{x})} - d > 2^{-j}$.
    Let $n$ and $m$ as above.
    Since $B$ is locally universal, for small $\epsilon > 0$ there exists $\overline{z}$ in $B$ such that $\norm{p_i(\overline{z})} < 2^{-m} - \epsilon$ for all $i$ and $\norm{z_k} < C_k + 2^{-m} - \epsilon$ for all $k$ and $\norm{q(\overline{z})} > d + 2^{-j} + \epsilon$.
    So, there exist rational points $\overline{r}$ of $B^\dagger$ such that $\norm{p_i(\overline{r})} < 2^{-m}$ for all $i$ and $\norm{r_k} < C_k + 2^{-m}$ for all $k$ and $\norm{q(\overline{r})} > d + 2^{-j}$.
\end{proof}

By Fact \ref{alec}, we immediately obtain the following:

\begin{cor}
    If there exists a simple locally universal \cstar-algebra $B$ with a c.e. presentation $B^\dagger$ (in particular, if KEP holds), then every computably weakly stable presentation $A^\dagger$ of a \cstar-algebra $A$ is computable.
\end{cor}

\section{A quantum complexity result}

A \textbf{nonlocal game with $n$ questions and $k$ answers} is a pair $\mathfrak G=(\pi,D)$, where $\pi$ is a probability distribution on $[n]$ and $D:[n]\times[n]\times[k]\times[k]\to \{0,1\}$ is called the \textbf{decision predicate for the game}.  Here, $[n]:=\{1,\ldots,n\}$ and analogously for $[k]$.  We also refer to the pair $(n,k)$ as the \textbf{dimensions} of $\mathfrak G$.  We view two players, henceforth referred to as Alice and Bob, playing $\mathfrak G$ as follows:  a pair of questions $(x,y)\in [n]\times [n]$ is randomly chosen according to $\pi$ and then Alice and Bob somehow respond with a pair of answers $(a,b)\in [k]\times [k]$; they win the game if $D(x,y,a,b)=1$ and otherwise they lose the game.

In order to describe their strategies for playing $\mathfrak G$, we need the notion of POVMs.  Recall that a \textbf{positive operator-valued measure} or \textbf{POVM} on a Hilbert space $\cal H$ is a finite collection $A_1,\ldots,A_k$ of positive operators on $\cal H$ such that $A_1+\cdots+A_k=I$.  We refer to $k$ as the \textbf{length} of the POVM.  More generally, one can use the same definition to define a POVM in any \cstar-algebra.

For each $k$, let $\varphi_k(X)$ denote the formula $$\max\left(\max_{1\leq i\leq k}\inf_{Z_i} \|Z_i^*Z_i-X_i\|,\|\sum_{i=1}^kX_i-I\|\right)$$ in the $k$ variables $X=(X_1,\ldots,X_k)$.  The following lemma is easy but will be used throughout the paper:

\begin{lem}\label{definable}
For each $\epsilon>0$ and $k\geq 1$, there is $\delta>0$ such that:  for any \cstar-algebra $C$ and any elements $A_1,\ldots,A_k$ from the unit ball of $C$, if $\varphi_k(A_1,\ldots,A_k)<\delta$, then there is a POVM $B_1,\ldots,B_k$ in $C$ such that $\max_{1\leq i\leq k}\|A_i-B_i\|<\epsilon$.
\end{lem}

We use POVMs to define strategies for nonlocal games.  First suppose that $\mathfrak G$ is a nonlocal game with dimensions $(n,k)$ and $C$ is a \cstar-algebra.  A \textbf{$\mathfrak G$-measurement in $C$} is a tuple $A:=(A^x)_{x\in [n]}$ of POVMs in $C$, each of which has length $k$.  Of course, the notion of a $\mathfrak G$-measurement in $C$ only depends on the dimensions of the nonlocal game, but the terminology will prove useful in the sequel.  Thus, corresponding to each possible question and answer pair $(x,a)\in [n]\times [k]$, we will have a positive element $A^x_a\in C$, and for each $x\in [n]$, we have $\sum_{a\in [k]}A^x_a=I$.

A \textbf{$\mathfrak G$-strategy in $C$} is a tuple $\sigma:=(A,B,\phi)$, where $A$ and $B$ are $\mathfrak G$-measurements in $C$ and $\phi\in S(C)$ is a state on $C$.  Given a $\mathfrak G$-strategy $\sigma$ in $C$ as above, we define the corresponding correlation matrix $p_\sigma\in [0,1]^{n^2k^2}$ by $p_{\sigma}(a,b|x,y)=\phi(A^x_a\bullet B^y_b)$, where, for any $A,B\in C$, we define $A\bullet B:=\frac{1}{2}(A^{1/2}BA^{1/2}+B^{1/2}AB^{1/2})$.\footnote{This notation seems to have first been considered by Ozawa in \cite{ozawa}.}  The intuition behind this definition is that if Alice and Bob play $\mathfrak G$ according to the strategy $\sigma$, then upon receiving the question pair $(x,y)$, they both measure their portion of the state $\phi$ using their POVMs $A^x$ and $B^y$; since we are not assuming that these measurements commute, we take the average of the results obtained from when Alice measures first and from when Bob measures first.  Consequently, $p_\sigma(a,b|x,y)$ is the probability that they answer the question pair $(x,y)$ with the answer pair $(a,b)$ when using the strategy $\sigma$.  

Of course, if each $A^x_a$ and $B^y_b$ commute, the above definition degenerates to the usual situation of calculating $p_\sigma(a,b|x,y)=\phi(A^x_aB^y_b)$ and we call the strategy $\sigma$ \textbf{commuting}.  Fix a positive real number $\delta>0$.  We call the strategy $\sigma$ \textbf{$\delta$-op-almost commuting} if $\sum_{a,b\in [k]}\|[A^x_a,B^y_b]\|<\delta$ for all $(x,y)\in [n]\times [n]$.  Note that the notion of a $\delta$-op-commuting strategy depends only on the pair of $\mathfrak G$-measurements, whence it makes sense to say that a pair of such measurements is a \textbf{$\delta$-op-almost commuting pair}.

Given a $\mathfrak G$-strategy $\sigma$ in $C$, we define the \textbf{value of $\mathfrak G$ when playing according to $\sigma$} to be the expected value Alice and Bob have of winning the game when playing according to $\sigma$, that is,
$$\val(\mathfrak G,\sigma):=\sum_{x,y}\pi(x,y)\sum_{a,b}D(x,y,a,b)p_\sigma(a,b|x,y).$$ In the sequel, it will behoove us to define, for every pair $(A,B)$ of $\mathfrak G$-measurements in $C$, the element $$\mathfrak G(A,B):=\sum_{x,y}\pi(x,y)\sum_{a,b}D(x,y,a,b)(A^x_a\bullet B^y_b).$$  With this notation, we have $\val(\mathfrak G,\sigma)=\phi(\mathfrak G(A,B))$.

If one considers the supremum of $\val(\mathfrak G,\sigma)$ as $\sigma$ ranges over all commuting $\mathfrak G$-strategies in $\mathcal B(\cal H)$, one obtains the \textbf{commuting value of $\mathfrak G$}, denoted $\val^{co}(\mathfrak G)$.  
Similarly, given $\delta>0$, we define the \textbf{$\delta,\op$-commuting value of $\mathfrak G$}, denoted $\val^{co}_{\delta,\op}(\mathfrak G)$, to be the supremum of $\val(\mathfrak G,\sigma)$ as $\sigma$ ranges over all $\delta$-op-almost commuting $\mathfrak G$-strategies in $\cal B(\cal H)$.  



Recall that a \textbf{language} (in the sense of complexity theory) is simply a subset of $2^{<\bb N}$, that is, is a set of finite sequences of bits.
\begin{defn}
We say that a language $L$ belongs to $\mip^{co}$ if there is an efficient mapping from sequences of bits $z$ to nonlocal games $\mathfrak G_z$ such that:
\begin{itemize}
    \item If $z\in L$, then $\val^{co}(\mathfrak G_z)=1$.
    \item If $z\notin L$, then $\val^{co}(\mathfrak G_z)\leq \frac{1}{2}$.
\end{itemize}
\end{defn}

In \cite{NPA08}, it was shown that if $L$ belongs to $\mip^{co}$, then $L$ belongs to the complexity class coRE of sets whose complement is c.e. and it was asked if this inclusion is in fact an equality.  The main result of this section shows that the analogous question has a negative answer for a suitably defined almost-commuting version of the class $\mip^{co}$:

\begin{defn}
Fix a computable function $\delta:\bb N\to [0,1]$.  We say that a language $L$ belongs to $\mip^{co}_{\delta,\op}$ if there is an efficient mapping from sequences of bits $z$ to nonlocal games $\mathfrak G_z$ such that:
\begin{itemize}
    \item If $z\in L$, then $\val^{co}_{\delta(|z|),\op}(\mathfrak G_z)=1$.
    \item If $z\notin L$, then $\val^{co}_{\delta(|z|),\op}(\mathfrak G_z)\leq \frac{1}{2}$.
\end{itemize}
\end{defn}

\begin{thm}\label{KEP4}
If there is a locally universal \cstar-algebra $C$ that has a computable presentation (in particular, if KEP holds), then for every computable function $\delta:\bb N\to [0,1]$, every language in $\mip^{co}_{\delta,\op}$ is recursively enumerable.
\end{thm}

\begin{proof}
Suppose that $C$ is a locally universal \cstar-algebra with a computable presentation $C^\#$.  Fix a computable function $\delta:\bb N\to [0,1]$ and suppose that $L$ belongs to $\mip^{co}_{\delta,\op}$.  Here is the algorithm that shows that $L$ is recursively enumerable.  Suppose that one inputs the sequence of bits $z$.  Set $\mathfrak G:=\mathfrak G_{z}$.  Start enumerating pairs of $\delta(|z|)$-almost commuting $\mathfrak G$-measurements $(A,B)$ in $C$ that consist only of rational points of $C^\#$; this is possible since the presentation is computable.  If $(A,B)$ is such a $\delta(|z|)$-op-almost commuting pair, approximate $\|\mathfrak G(A,B)\|$ to within error $\frac{1}{4}$.  If this approximation exceeds $\frac{1}{2}$, then declare that $z\in L$.

Here is why the algorithm works.  We first show that if $z\in L$, then the algorithm will tell us so.  As above, set $\mathfrak G:=\mathfrak G_z$.  Fix $\epsilon>0$ small enough and let $\sigma:=(A,B,\phi)$ be a $\delta(|z|)$-$\op$-almost commuting $\mathfrak G$-strategy in $\cal B(\cal H)$ such that $\val(\mathfrak G,\sigma)>1-\epsilon$.  It follows that $\|\mathfrak G(A,B)\|>1-\epsilon$. Let $D$ be the \cstar-algebra generated by the coordinates of $A$ and $B$ and consider an embedding of $D$ into $C^\u$.  It follows from Lemma \ref{definable} that there are $\delta(|z|)$-$\op$-almost commuting $\mathfrak G$-measurements $\bar A$ and $\bar B$ in $C$ for which  $\|\mathfrak G(\bar A,\bar B)\|>1-2\epsilon$.  Without loss of generality, one can assume that the coordinates of $\bar A$ and $\bar B$ are rational points of $C^\#$.  If $\epsilon$ is small enough, then approximating $\|\mathfrak G(\bar A,\bar B)\|$ to within error $\frac{1}{4}$ will exceed $\frac{1}{2}$.  Thus, our algorithm will eventually tell us that $z\in L$.

We now check that the algorithm makes no mistakes, that is, if the algorithm tells us that $z\in L$, then in fact $z$ does belong to $L$.  If the algorithm tells us that $(A,B)$ is a $\delta(|z|)$-$\op$-almost commuting pair of $\mathfrak G$-measurements in $C$ for which $\|\mathfrak G(A,B)\|>\frac{1}{2}$, then there will be some state $\phi$ on $C$ such that $\phi(\mathfrak G(A,B))>\frac{1}{2}$. Setting $\sigma:=(A,B,\phi)$, we see that $\val^{co}_{\op,\delta}(\mathfrak G)>\frac{1}{2}$.  Consequently, $z\in L$, as desired.
\end{proof}

We note two things about Theorem \ref{KEP4}.  First, as is the case for $\mip^{co}$, one can show that every language in $\mip^{co}_{\delta,\op}$ is coRE (regardless of the truth of KEP), whence the theorem implies that every language in $\mip^{co}_{\delta,\op}$ is actually decidable provided that KEP holds.  Second, the class $\mip^{co}_{\delta,\op}$ differs from the class $\mip^{co}_{\delta}$ introduced by Coudron and Slofstra in \cite{CS}; in particular, every language in their class is decidable without any KEP assumption. 


\end{document}